\documentclass{amsart}
\usepackage{amstext,amssymb,amsthm,amsopn,newlfont,graphpap,graphics,graphicx,mathrsfs,enumitem}
\allowdisplaybreaks
\usepackage[parfill]{parskip}
\usepackage[noadjust]{cite}
\usepackage{epigraph}
\usepackage[colorlinks=true,
            linkcolor=red,
            urlcolor=blue,
            citecolor=magenta]{hyperref}
\usepackage{color}
\usepackage{mathrsfs}
\allowdisplaybreaks
\theoremstyle{plain}
\newtheorem{thm}{Theorem}[section]
\newtheorem*{thm*}{Theorem}
\newtheorem{prop}{Proposition}[section]
\newtheorem*{prop*}{Proposition}

\newtheorem*{cor*}{Corollary}
\newtheorem{lem}{Lemma}[section]
\newtheorem*{lem*}{Lemma}
\theoremstyle{definition}
\newtheorem{defn}{Definition}[section]
\newtheorem*{defn*}{Definition}

\newtheorem*{exmp*}{Example}

\newtheorem*{exmps*}{Examples}

\newtheorem*{rem*}{Remark}
\newtheorem{rems}{Remarks}[section]
\newtheorem*{rems*}{Remarks}

\newtheorem*{note*}{Note}
\newcommand{\N}{{\mathbb N}}
\newcommand{\Z}{{\mathbb Z}}
\newcommand{\C}{{\mathbb C}}

\DeclareMathOperator{\Rep}{Re\,}

\DeclareMathOperator{\dist}{dist}
\DeclareMathOperator{\spa}{span}
\begin{document}
\title[On the Gevrey ultradifferentiability of weak solutions]
{On the Gevrey ultradifferentiability\\ 
of weak solutions\\
of an abstract evolution equation\\
with a scalar type spectral operator\\
of orders less than one}
\author[Marat V. Markin]{Marat V. Markin}
\address{
Department of Mathematics\newline
California State University, Fresno\newline
5245 N. Backer Avenue, M/S PB 108\newline
Fresno, CA 93740-8001, USA
}
\email{mmarkin@csufresno.edu}
\keywords{Weak solution, scalar type spectral operator, Gevrey classes}
\subjclass[2010]{Primary 34G10, 47B40, 30D15; Secondary 47B15, 47D06, 47D60, 30D60}
\begin{abstract}
It is shown that, if all \textit{weak solutions} of the evolution equation
\begin{equation*}
y'(t)=Ay(t),\ t\ge 0,
\end{equation*}
with a scalar type spectral operator $A$ in a complex Banach space are Gevrey ultradifferentiable of orders less than one, then the operator $A$ is necessarily bounded.
\end{abstract}

\maketitle
\epigraph{\textit{Too much of a good thing can be wonderful.}}{Mae West}
\section[Introduction]{Introduction}

In \cite{Markin2011,Markin2018(3),Markin2018(4)}, found are characterizations of the strong differentiability and Gevrey ultradifferentiability of order $\beta\ge 1$, in particular \textit{analyticity} and \textit{entireness}, on $[0,\infty)$ and $(0,\infty)$ of all \textit{weak solutions} of the evolution equation
\begin{equation}\label{1}
y'(t)=Ay(t),\ t\ge 0, 
\end{equation}
with a scalar type spectral operator $A$ in a complex Banach space.

As is shown by {\cite[Theorem $4.1$]{Markin2018(3)}} (see also {\cite[Corollary $4.1$]{Markin2018(3)}}), all weak solutions of equation \eqref{1} can be \textit{entire} vector functions, i.e., belong to the
\textit{first-order} Beurling type Gevrey class 
${\mathscr E}^{(1)}([0,\infty),X)$ (see Preliminaries), while the operator $A$ is \textit{unbounded}, e.g., when $A$ is a semibounded below \textit{self-adjoint} operator in a complex Hilbert space (see {\cite[Corollary $4.1$]{Markin2001(1)}}). This remarkable fact contrasts the situation when, in \eqref{1}, a closed densely defined linear operator $A$ generates a $C_0$-semigroup, in which case the strong differentiability of all weak solutions of \eqref{1} at $0$ alone immediately implies \textit{boundedness} for $A$ (cf. \cite{Engel-Nagel}, see also \cite{Markin2002(2)}).

It remains to examine whether all weak solutions of equation \eqref{1} with a scalar type spectral operator $A$ in a complex Banach space can belong to the Gevrey classes of orders \textit{less than one} (not necessarily to the same one) with $A$ remaining \textit{unbounded}.

In this paper, developing the results of
\cite{Markin2011,Markin2018(3),Markin2018(4)}, we show that an unbounded scalar type spectral operator $A$ in a complex Banach space cannot sustain the strong Gevrey ultradifferentiability of all weak solutions of equation \eqref{1} for orders less than one, i.e., that
imposing on all the weak solutions along with the entireness requirement, certain growth at infinity conditions (see Preliminaries) necessarily makes the operator $A$ \textit{bounded}. Thus, we generalize the corresponding results for equation \eqref{1} with a \textit{normal operator} $A$ in a complex Hilbert space found in \cite{Markin2001(2)}.

\begin{defn}[Weak Solution]\label{ws}\ \\
Let $A$ be a densely defined closed linear operator in a Banach space $(X,\|\cdot\|)$. A strongly continuous vector function $y:[0,\infty)\rightarrow X$ is called a {\it weak solution} of equation \eqref{1} if, for any $g^* \in D(A^*)$,
\begin{equation*}
\dfrac{d}{dt}\langle y(t),g^*\rangle = \langle y(t),A^*g^* \rangle,\ t\ge 0,
\end{equation*}
where $D(\cdot)$ is the \textit{domain} of an operator, $A^*$ is the operator {\it adjoint} to $A$, and $\langle\cdot,\cdot\rangle$ is the {\it pairing} between
the space $X$ and its dual $X^*$ (cf. \cite{Ball}).
\end{defn}

\begin{rems}\label{remsws}\
\begin{itemize}
\item Due to the \textit{closedness} of $A$, the weak solution of \eqref{1} can be equivalently defined to be a strongly continuous vector function $y:[0,\infty)\mapsto X$ such that, for all $t\ge 0$,
\begin{equation*}
\int_0^ty(s)\,ds\in D(A)\ \text{and} \ y(t)=y(0)+A\int_0^ty(s)\,ds
\end{equation*}
and is also called a \textit{mild solution} (cf. {\cite[Ch. II, Definition 6.3]{Engel-Nagel}}, {\cite[Preliminaries]{Markin2018(2)}}).
\item Such a notion of \textit{weak solution}, which need not be differentiable in the strong sense, generalizes that of \textit{classical} one, strongly differentiable on $[0,\infty)$ and satisfying the equation in the traditional plug-in sense, the classical solutions being precisely the weak ones strongly differentiable on $[0,\infty)$.
\item When a closed densely defined linear operator $A$
in a complex Banach space $X$ generates a $C_0$-semigroup $\left\{T(t) \right\}_{t\ge 0}$ of  bounded linear operators (see, e.g., \cite{Hille-Phillips,Engel-Nagel}), i.e., the associated \textit{abstract Cauchy problem} (\textit{ACP})
\begin{equation}\label{ACP}
\begin{cases}
y'(t)=Ay(t),\ t\ge 0,\\
y(0)=f
\end{cases}
\end{equation}
is \textit{well-posed} (cf. {\cite[Ch. II, Definition 6.8]{Engel-Nagel}}), the weak solutions of equation \eqref{1} are the orbits
\begin{equation}\label{semigroup}
y(t)=T(t)f,\ t\ge 0,
\end{equation}
with $f\in X$ {\cite[Ch. II, Proposition 6.4]{Engel-Nagel}} (see also {\cite[Theorem]{Ball}}), whereas the classical ones are those with $f\in D(A)$
(see, e.g., {\cite[Ch. II, Proposition 6.3]{Engel-Nagel}}). 
\item In our discourse, the associated \textit{ACP} need not be \textit{well-posed}, i.e., the scalar type spectral operator $A$ need not generate a $C_0$-semigroup (cf. \cite{Markin2002(2)}).
\end{itemize} 
\end{rems} 

\section[Preliminaries]{Preliminaries}

Here, for the reader's convenience, we outline certain essential preliminaries.

\subsection{Scalar Type Spectral Operators}\ 

Henceforth, unless specified otherwise, $A$ is supposed to be a {\it scalar type spectral operator} in a complex Banach space $(X,\|\cdot\|)$ with strongly $\sigma$-additive \textit{spectral measure} (the \textit{resolution of the identity}) $E_A(\cdot)$ assigning to each Borel set $\delta$ of the complex plane $\C$ a projection operator $E_A(\delta)$ on $X$ and having the operator's \textit{spectrum} $\sigma(A)$ as its {\it support} \cite{Dunford1954,Survey58,Dun-SchIII}.

Observe that, in a complex finite-dimensional space, 
the scalar type spectral operators are all linear operators on the space, for which there is an \textit{eigenbasis} (see, e.g., \cite{Survey58,Dun-SchIII}) and, in a complex Hilbert space, the scalar type spectral operators are precisely all those that are similar to the {\it normal} ones \cite{Wermer}.

Associated with a scalar type spectral operator in a complex Banach space is the {\it Borel operational calculus} analogous to that for a \textit{normal operator} in a complex Hilbert space \cite{Survey58,Dun-SchII,Dun-SchIII,Plesner}, which assigns to any Borel measurable function $F:\sigma(A)\to \C$ a scalar type spectral operator
\begin{equation*}
F(A):=\int\limits_{\sigma(A)} F(\lambda)\,dE_A(\lambda)
\end{equation*}
defined as follows:
\[
F(A)f:=\lim_{n\to\infty}F_n(A)f,\ f\in D(F(A)),\
D(F(A)):=\left\{f\in X\middle| \lim_{n\to\infty}F_n(A)f\ \text{exists}\right\},
\]
where
\begin{equation*}
F_n(\cdot):=F(\cdot)\chi_{\{\lambda\in\sigma(A)\,|\,|F(\lambda)|\le n\}}(\cdot),
\ n\in\N,
\end{equation*}
($\chi_\delta(\cdot)$ is the {\it characteristic function} of a set $\delta\subseteq \C$, $\N:=\left\{1,2,3,\dots\right\}$ is the set of \textit{natural numbers}) and
\begin{equation*}
F_n(A):=\int\limits_{\sigma(A)} F_n(\lambda)\,dE_A(\lambda),\ n\in\N,
\end{equation*}
are {\it bounded} scalar type spectral operators on $X$ defined in the same manner as for a {\it normal operator} (see, e.g., \cite{Dun-SchII,Plesner}).

In particular,
\begin{equation*}
A^n=\int\limits_{\sigma(A)} \lambda^n\,dE_A(\lambda),\ n\in\Z_+,
\end{equation*}
($\Z_+:=\left\{0,1,2,\dots\right\}$ is the set of \textit{nonnegative integers}, $A^0:=I$, $I$ is the \textit{identity operator} on $X$) and
\begin{equation}\label{exp}
e^{zA}:=\int\limits_{\sigma(A)} e^{z\lambda}\,dE_A(\lambda),\ z\in\C.
\end{equation}

The properties of the {\it spectral measure} and {\it operational calculus}, exhaustively delineated in \cite{Survey58,Dun-SchIII}, underlie the entire subsequent discourse. Here, we underline a few facts of particular importance.

Due to its {\it strong countable additivity}, the spectral measure $E_A(\cdot)$ is {\it bounded} \cite{Dun-SchI,Dun-SchIII}, i.e., there is such an $M\ge 1$ that, for any Borel set $\delta\subseteq \C$,
\begin{equation}\label{bounded}
\|E_A(\delta)\|\le M.
\end{equation}

Observe that the notation $\|\cdot\|$ is used here to designate the norm in the space $L(X)$ of all bounded linear operators on $X$. We adhere to this rather conventional economy of symbols in what follows also adopting the same notation for the norm in the dual space $X^*$.

For any $f\in X$ and $g^*\in X^*$, the \textit{total variation measure} $v(f,g^*,\cdot)$ of the complex-valued Borel measure $\langle E_A(\cdot)f,g^* \rangle$ is a {\it finite} positive Borel measure with
\begin{equation}\label{tv}
v(f,g^*,\C)=v(f,g^*,\sigma(A))\le 4M\|f\|\|g^*\|
\end{equation}
(see, e.g., \cite{Markin2004(1),Markin2004(2)}).

Also (Ibid.), for a Borel measurable function $F:\C\to \C$, $f\in D(F(A))$, $g^*\in X^*$, and a Borel set $\delta\subseteq \C$,
\begin{equation}\label{cond(ii)}
\int\limits_\delta|F(\lambda)|\,dv(f,g^*,\lambda)
\le 4M\|E_A(\delta)F(A)f\|\|g^*\|.
\end{equation}
In particular, for $\delta=\sigma(A)$,
\begin{equation}\label{cond(i)}
\int\limits_{\sigma(A)}|F(\lambda)|\,d v(f,g^*,\lambda)\le 4M\|F(A)f\|\|g^*\|.
\end{equation}

Observe that the constant $M\ge 1$ in \eqref{tv}--\eqref{cond(i)} is from 
\eqref{bounded}.

Further, for a Borel measurable function $F:\C\to [0,\infty)$, a Borel set $\delta\subseteq \C$, a sequence $\left\{\Delta_n\right\}_{n=1}^\infty$ 
of pairwise disjoint Borel sets in $\C$, and 
$f\in X$, $g^*\in X^*$,
\begin{equation}\label{decompose}
\int\limits_{\delta}F(\lambda)\,dv(E_A(\cup_{n=1}^\infty \Delta_n)f,g^*,\lambda)
=\sum_{n=1}^\infty \int\limits_{\delta\cap\Delta_n}F(\lambda)\,dv(E_A(\Delta_n)f,g^*,\lambda).
\end{equation}

Indeed, since, for any Borel sets $\delta,\sigma\subseteq \C$,
\begin{equation*}
E_A(\delta)E_A(\sigma)=E_A(\delta\cap\sigma)
\end{equation*}
\cite{Survey58,Dun-SchIII}, 
for the total variation,
\begin{equation*}
v(E_A(\delta)f,g^*,\sigma)=v(f,g^*,\delta\cap\sigma).
\end{equation*}

Whence, due to the {\it nonnegativity} of $F(\cdot)$ (see, e.g., \cite{Halmos}),
\begin{multline*}
\int\limits_\delta F(\lambda)\,dv(E_A(\cup_{n=1}^\infty \Delta_n)f,g^*,\lambda)
=\int\limits_{\delta\cap\cup_{n=1}^\infty \Delta_n}F(\lambda)\,dv(f,g^*,\lambda)
\\
\ \
=\sum_{n=1}^\infty \int\limits_{\delta\cap\Delta_n}F(\lambda)\,dv(f,g^*,\lambda)
=\sum_{n=1}^\infty \int\limits_{\delta\cap\Delta_n}F(\lambda)\,dv(E_A(\Delta_n)f,g^*,\lambda).
\hfill
\end{multline*}

The following statement, allowing to characterize the domains of Borel measurable functions of a scalar type spectral operator in terms of positive Borel measures, is fundamental for our discourse.

\begin{prop}[{\cite[Proposition $3.1$]{Markin2002(1)}}]\label{prop}\ \\
Let $A$ be a scalar type spectral operator in a complex Banach space $(X,\|\cdot\|)$ with spectral measure $E_A(\cdot)$ and $F:\sigma(A)\to \C$ be a Borel measurable function. Then $f\in D(F(A))$ iff
\begin{enumerate}[label={(\roman*)}]
\item for each $g^*\in X^*$, 
$\displaystyle \int\limits_{\sigma(A)} |F(\lambda)|\,d v(f,g^*,\lambda)<\infty$ and
\item $\displaystyle \sup_{\{g^*\in X^*\,|\,\|g^*\|=1\}}
\int\limits_{\{\lambda\in\sigma(A)\,|\,|F(\lambda)|>n\}}
|F(\lambda)|\,dv(f,g^*,\lambda)\to 0,\ n\to\infty$,
\end{enumerate}
where $v(f,g^*,\cdot)$ is the total variation measure of $\langle E_A(\cdot)f,g^* \rangle$.
\end{prop} 

The succeeding key theorem provides a description of the weak solutions of equation \eqref{1} with a scalar type spectral operator $A$ in a complex Banach space.

\begin{thm}[{\cite[Theorem $4.2$]{Markin2002(1)}}]\label{GWS}\ \\
Let $A$ be a scalar type spectral operator in a complex Banach space $(X,\|\cdot\|)$. A vector function $y:[0,\infty) \to X$ is a weak solution 
of equation \eqref{1} iff there is an $\displaystyle f \in \bigcap_{t\ge 0}D(e^{tA})$ such that
\begin{equation}\label{expf}
y(t)=e^{tA}f,\ t\ge 0,
\end{equation}
the operator exponentials understood in the sense of the Borel operational calculus (see \eqref{exp}).
\end{thm}

\begin{rems}\
\begin{itemize}
\item Theorem \ref{GWS} generalizing {\cite[Theorem $3.1$]{Markin1999}}, its counterpart for a normal operator $A$ in a complex Hilbert space, in particular,
implies
\begin{itemize}
\item that the subspace $\bigcap_{t\ge 0}D(e^{tA})$ of all possible initial values of the weak solutions of equation \eqref{1} is the largest permissible for the exponential form given by \eqref{expf}, which highlights the naturalness of the notion of weak solution, and
\item that associated \textit{ACP} \eqref{ACP}, whenever solvable,  is solvable \textit{uniquely}.
\end{itemize} 
\item Observe that the initial-value subspace $\bigcap_{t\ge 0}D(e^{tA})$ of equation \eqref{1}, containing the dense in $X$ subspace $\bigcup_{\alpha>0}E_A(\Delta_\alpha)X$, where
\begin{equation*}
\Delta_\alpha:=\left\{\lambda\in\C\,\middle|\,|\lambda|\le \alpha \right\},\ \alpha>0,
\end{equation*}
which coincides with the class ${\mathscr E}^{\{0\}}(A)$ of the \textit{entire} vectors of $A$ of \textit{exponential type} (see below), is \textit{dense} in $X$.
\item When a scalar type spectral operator $A$ in a complex Banach space generates a $C_0$-semigroup $\left\{T(t) \right\}_{t\ge 0}$, 
\[
T(t)=e^{tA}\ \text{and}\ D(e^{tA})=X,\ t\ge 0,
\]
\cite{Markin2002(2)}, and hence, Theorem \ref{GWS} is consistent with the well-known description of the weak solutions for this setup (see \eqref{semigroup}).
\end{itemize} 
\end{rems} 

Subsequently, the frequent terms {\it ``spectral measure"} and {\it ``operational calculus"} are abbreviated to {\it s.m.} and {\it o.c.}, respectively.

\subsection{Gevrey Classes of Functions}\

\begin{defn}[Gevrey Classes of Functions]\ \\
Let $(X,\|\cdot\|)$ be a (real or complex) Banach space, $C^\infty(I,X)$ be the space of all $X$-valued functions strongly infinite differentiable on an interval $I\subseteq (-\infty,\infty)$, and $0\le \beta<\infty$.

The following subspaces of $C^\infty(I,X)$
\begin{align*}
{\mathscr E}^{\{\beta\}}(I,X):=\bigl\{g(\cdot)\in C^{\infty}(I, X) \bigm |&
\forall [a,b] \subseteq I\ \exists \alpha>0\ \exists c>0:
\\
&\max_{a \le t \le b}\|g^{(n)}(t)\| \le c\alpha^n [n!]^\beta,
\ n\in\Z_+\bigr\},\\
{\mathscr E}^{(\beta)}(I,X):= \bigl\{g(\cdot) \in C^{\infty}(I,X) \bigm |& 
\forall [a,b] \subseteq I\ \forall \alpha > 0 \ \exists c>0:
\\
&\max_{a \le t \le b}\|g^{(n)}(t)\| \le c\alpha^n [n!]^\beta,
\ n\in\Z_+\bigr\},
\end{align*}
are called the {\it $\beta$th-order Gevrey classes} of strongly ultradifferentiable vector functions on $I$ of {\it Roumieu} and {\it Beurling type}, respectively (see, e.g., \cite{Gevrey,Komatsu1,Komatsu2,Komatsu3}).
\end{defn}

\begin{rems}\
\begin{itemize}
\item In view of {\it Stirling's formula}, the 
sequence $\left\{ [n!]^\beta\right\}_{n=0}^\infty$ can be replaced with
$\left\{ n^{\beta n}\right\}_{n=0}^\infty$.
\item For $0\le\beta<\beta'<\infty$, the inclusions
\begin{equation*}
{\mathscr E}^{(\beta)}(I,X)\subseteq{\mathscr E}^{\{\beta\}}(I,X)
\subseteq {\mathscr E}^{(\beta')}(I,X)\subseteq
{\mathscr E}^{\{\beta'\}}(I,X)\subseteq C^{\infty}(I,X)
\end{equation*}
hold.
\item For $1<\beta<\infty$, the Gevrey classes
are \textit{non-quasianalytic} (see, e.g., \cite{Komatsu2}).
\item For $\beta=1$, ${\mathscr E}^{\{1\}}(I,X)$ 
is the class of all {\it analytic} on $I$, i.e., {\it analytically continuable} into complex neighborhoods of $I$, vector functions and ${\mathscr E}^{(1)}(I,X)$ is the class of all {\it entire}, i.e., allowing {\it entire} continuations, vector functions \cite{Mandel}.
\item For $0\le\beta<1$, the Gevrey class ${\mathscr E}^{\{\beta\}}(I,X)$ (${\mathscr E}^{(\beta)}(I,X)$) consists of all functions $g(\cdot)\in {\mathscr E}^{(1)}(I,X)$ such that, for some (any) $\gamma>0$, there is an $M>0$ for which
\begin{equation}\label{order}
\|g(z)\|\le Me^{\gamma|z|^{1/(1-\beta)}},\ z\in \C,
\end{equation}
\cite{Markin2001(2)}. In particular,
for $\beta=0$, ${\mathscr E}^{\{0\}}(I,X)$ and ${\mathscr E}^{(0)}(I,X)$ are the classes of entire vector functions of \textit{exponential} and \textit{minimal exponential type}, respectively (see, e.g., \cite{Levin}).
\end{itemize} 
\end{rems} 

\subsection{Gevrey Classes of Vectors}\

One can consider the Gevrey classes in a more general sense. 

\begin{defn}[Gevrey Classes of Vectors]\ \\
Let $(A,D(A))$ be a densely defined closed linear operator in a (real or complex) Banach space $(X,\|\cdot\|)$, $0\le \beta<\infty$, and
\begin{equation*}
C^{\infty}(A):=\bigcap_{n=0}^{\infty}D(A^n)
\end{equation*}
be the subspace of infinite differentiable vectors of $A$.

The following subspaces of $C^{\infty}(A)$
\begin{align*}
{\mathscr E}^{\{\beta\}}(A)&:=\left\{x\in C^{\infty}(A)\, \middle |\, 
\exists \alpha>0\ \exists c>0:
\|A^nx\| \le c\alpha^n [n!]^\beta,\ n\in\Z_+ \right\},\\
{\mathscr E}^{(\beta)}(A)&:=\left\{x \in C^{\infty}(A)\, \middle|\,\forall \alpha > 0 \ \exists c>0:
\|A^nx\| \le c\alpha^n [n!]^\beta,\ n\in\Z_+ \right\}
\end{align*}
are called the \textit{$\beta$th-order Gevrey classes} of ultradifferentiable vectors of $A$ of \textit{Roumieu} and \textit{Beurling type}, respectively (see, e.g., \cite{GorV83,Gor-Knyaz,book}).
\end{defn}

\begin{rems}\
\begin{itemize}
\item In view of {\it Stirling's formula}, the 
sequence $\left\{ [n!]^\beta\right\}_{n=0}^\infty$ can be replaced with
$\left\{ n^{\beta n}\right\}_{n=0}^\infty$.
\item For $0\le\beta<\beta'<\infty$, the inclusions
\begin{equation}\label{incl1}
{\mathscr E}^{(\beta)}(A)\subseteq{\mathscr E}^{\{\beta\}}(A)
\subseteq {\mathscr E}^{(\beta')}(A)\subseteq
{\mathscr E}^{\{\beta'\}}(A)\subseteq C^{\infty}(A)
\end{equation}
hold.
\item In particular, ${\mathscr E}^{\{1\}}(A)$ and ${\mathscr E}^{(1)}(A)$ are the classes of {\it analytic} and {\it entire} vectors of $A$, respectively \cite{Goodman,Nelson} and ${\mathscr E}^{\{0\}}(A)$ and ${\mathscr E}^{(0)}(A)$ are the classes of \textit{entire} vectors of $A$ of \textit{exponential} and \textit{minimal exponential type}, respectively (see, e.g., \cite{Radyno1983(1),Gor-Knyaz}).
\item In view of the \textit{closedness} of $A$, it is easily seen that the class ${\mathscr E}^{(1)}(A)$ forms the subspace of initial values in $X$ generating the (classical) solutions of \eqref{1}, which are entire vector functions represented by the power series
\begin{equation}\label{powser}
\sum_{n=0}^\infty \dfrac{t^n}{n!}A^nf,\ t\ge 0,
f\in {\mathscr E}^{(1)}(A),
\end{equation}
the classes ${\mathscr E}^{\{\beta\}}(A)$ and ${\mathscr E}^{(\beta)}(A)$ with $0\le\beta<1$ being the subspaces of such initial values for which the solutions satisfy growth estimate \eqref{order} with some (any) $\gamma>0$ and some $M>0$, respectively (cf. \cite{Levin}).
\end{itemize} 
\end{rems} 

As is shown in \cite{GorV83} (see also \cite{Gor-Knyaz,book}), if $0<\beta<\infty$, for a {\it normal operator} $A$ in a complex Hilbert space,
\begin{equation}\label{GC}
{\mathscr E}^{\{\beta\}}(A)=\bigcup_{t>0} D(e^{t|A|^{1/\beta}})\ \text{and}\ 
{\mathscr E}^{(\beta)}(A)=\bigcap_{t>0} D(e^{t|A|^{1/\beta}}),
\end{equation}
the operator exponentials $e^{t|A|^{1/\beta}}$, $t>0$, understood in the sense of the Borel operational calculus (see, e.g., \cite{Dun-SchII,Plesner}).

In \cite{Markin2004(2),Markin2015}, descriptions \eqref{GC} are extended  to \textit{scalar type spectral operators} in a complex Banach space, in which form they are basic for our discourse. In \cite{Markin2015}, similar nature descriptions of the classes ${\mathscr E}^{\{0\}}(A)$ and ${\mathscr E}^{(0)}(A)$ ($\beta=0$), known for a normal operator $A$ in a complex Hilbert space (see, e.g., \cite{Gor-Knyaz}), are also generalized to scalar type spectral operators in a complex Banach space. In particular {\cite[Theorem $5.1$]{Markin2015}},
\[
{\mathscr E}^{\{0\}}(A)=\bigcup_{\alpha>0}E_A(\Delta_\alpha)X,
\] 
where
\begin{equation*}
\Delta_\alpha:=\left\{\lambda\in\C\,\middle|\,|\lambda|\le \alpha \right\},\ \alpha>0.
\end{equation*}

We also need the following characterization of a particular weak solution's of equation \eqref{1} with a scalar type spectral operator $A$ in a complex Banach space being strongly Gevrey ultradifferentiable on a subinterval $I$ of $[0,\infty)$ proved in \cite{Markin2018(2)}.

\begin{prop}[{\cite[Proposition $3.1$]{Markin2018(2)}}]\label{particular}
Let $A$ be a scalar type spectral operator in a complex Banach space $(X,\|\cdot\|)$, $0\le \beta<\infty$, and $I$ be a subinterval of $[0,\infty)$. Then the restriction of
a weak solution $y(\cdot)$ of equation \eqref{1} to $I$ belongs to the Gevrey class ${\mathscr E}^{\{\beta\}}(I,X)$
\textup{(${\mathscr E}^{(\beta)}(I,X)$)} iff, for each $t\in I$, 
\begin{equation*}
y(t) \in {\mathscr E}^{\{\beta\}}(A)
\ \textup{(${\mathscr E}^{(\beta)}(A)$, respectively)},
\end{equation*}
in which case
\begin{equation*}
y^{(n)}(t)=A^ny(t),\ n\in \N,t\in I.
\end{equation*}
\end{prop}

\section{One Lemma}

The following lemma generalizes {\cite[Lemma $4.1$]{Markin2001(2)}}, its counterpart for a normal operator in a complex Hilbert space and, besides being an interesting result by itself, is necessary for proving our main statement. 

\begin{lem}\label{bbounded}
Let $A$ be a scalar type spectral operator in a complex Banach space $(X,\|\cdot\|)$ with spectral measure $E_A(\cdot)$ and $0<\beta'<\infty$. If 
\begin{equation*}
\bigcup_{0\le \beta<\beta'}{\mathscr E}^{\{\beta\}}(A)={\mathscr E}^{(\beta')}(A),
\end{equation*}
then the operator $A$ is bounded.
\end{lem}

\begin{proof}\quad
First, observe that, in view of inclusions \eqref{incl1},
for any $\beta'>0$,
\begin{equation}\label{incl2}
\bigcup_{0\le \beta<\beta'}{\mathscr E}^{\{\beta\}}(A)\subseteq {\mathscr E}^{(\beta')}(A).
\end{equation}

Let us prove the statement \textit{by contrapositive} assuming $A$ to be \textit{unbounded}.

The operator $A$ being scalar type spectral, this assumption implies that the spectrum $\sigma(A)$ of $A$ is an \textit{unbounded} set in the complex plane $\C$
\cite{Survey58,Dun-SchIII}. Hence, the points of the spectrum can be found in \textit{infinitely many} semi-open annuli of the form
\begin{equation*}
\delta_n:=\left\{ \lambda \in \C\,\middle|\,n\le |\lambda|<n+1 \right\},\ n\in\N,
\end{equation*}
i.e., there is a sequence of natural numbers
$\{n(k)\}_{k=1}^\infty$ such that
\begin{equation}\label{seq}
k\le n(k)<n(k+1),\ k\in\N,
\end{equation}
and, for each $k\in\N$, there is a
\begin{equation*}
\lambda_k\in \delta_{n(k)}\cap \sigma(A)\neq \emptyset.
\end{equation*}

Setting $\varepsilon_0:=1$, for each $k\in\N$,
one can choose an 
\begin{equation}\label{eps}
0<\varepsilon_k<\min\left(1/k,\varepsilon_{k-1}\right)
\end{equation}
such that
\begin{equation}\label{or}
\lambda_k\in \Delta_k:=\left\{\lambda \in \C\, \middle|\, n(k)-\varepsilon_k <|\lambda|<n(k)+1
-\varepsilon_k\right\}.
\end{equation}

Since
\[
n(k)<n(k+1)\ \text{and}\ \varepsilon_{k+1}<\varepsilon_{k},\ k\in\N,
\]
we have:
\[
n(k)+1-\varepsilon_{k}<n(k+1)-\varepsilon_{k+1},\ k\in \N,
\]
and hence, the \textit{open} annuli $\Delta_k$, $k\in\N$, are {\it pairwise disjoint}.

Whence, by the properties of the {\it s.m.}, 
\begin{equation*}
E_A(\Delta_i)E_A(\Delta_j)=0,\ i\neq j,
\end{equation*}
where $0$ stands for the \textit{zero operator} on $X$.

Observe also, that the subspaces $E_A(\Delta_k)X$, $k\in \N$, are \textit{nontrivial} since
\[
\Delta_k \cap \sigma(A)\neq \emptyset,\ k\in\N,
\]
with $\Delta_k$ being an \textit{open set} in $\C$. 

By choosing a unit vector $e_k\in E_A(\Delta_k)X$ for each $k\in\N$, we obtain a sequence 
$\left\{e_n\right\}_{n=1}^\infty$ in $X$ such that
\begin{equation}\label{ortho1}
\|e_k\|=1,\ k\in\N,\ \text{and}\ E_A(\Delta_i)e_j=\delta_{ij}e_j,\ i,j\in\N,
\end{equation}
where $\delta_{ij}$ is the \textit{Kronecker delta}.

As is easily seen, \eqref{ortho1} implies that the vectors $e_k$, $k\in\N$, are \textit{linearly independent}.

Furthermore, there is an $\varepsilon>0$ such that
\begin{equation}\label{dist1}
d_k:=\dist\left(e_k,\spa\left(\left\{e_i\,|\,i\in\N,\ i\neq k\right\}\right)\right)\ge\varepsilon,\ k\in\N.
\end{equation}

Indeed, the opposite implies the existence of a subsequence $\left\{d_{k(m)}\right\}_{m=1}^\infty$ such that
\begin{equation*}
d_{k(m)}\to 0,\ m\to\infty.
\end{equation*}

Then, by selecting a vector
\[
f_{k(m)}\in 
\spa\left(\left\{e_i\,|\,i\in\N,\ i\neq k(m)\right\}\right),\ m\in\N,
\] 
such that 
\[
\|e_{k(m)}-f_{k(m)}\|<d_{k(m)}+1/m,\ m\in\N,
\]
we arrive at
\begin{multline*}
1=\|e_{k(m)}\|
\hfill
\text{since, by \eqref{ortho1}, 
$E_A(\Delta_{k(m)})f_{k(m)}=0$;}
\\
\shoveleft{
=\|E_A(\Delta_{k(m)})(e_{k(m)}-f_{k(m)})\|\
\le \|E_A(\Delta_{k(m)})\|\|e_{k(m)}-f_{k(m)}\|
\hfill
\text{by \eqref{bounded};}
}\\
\ \
\le M\|e_{k(m)}-f_{k(m)}\|\le M\left[d_{k(m)}+1/m\right]
\to 0,\ m\to\infty,
\hfill
\end{multline*}
which is a \textit{contradiction} proving \eqref{dist1}. 

As follows from the {\it Hahn-Banach Theorem}, for any $k\in\N$, there is an $e^*_k\in X^*$ such that 
\begin{equation}\label{H-B1}
\|e_k^*\|=1,\ k\in\N,\ \text{and}\ \langle e_i,e_j^*\rangle=\delta_{ij}d_i,\ i,j\in\N.
\end{equation}

Let
\begin{equation*}
f:=\sum_{k=1}^\infty n(k)^{-(n(k)+1-\varepsilon_k)^{1/\beta'}}e_k\in X
\ \text{and}\ 
h:=\sum_{k=1}^\infty n(k)^{-\frac{1}{2}(n(k)+1-\varepsilon_k)^{1/\beta'}}e_k\in X,
\end{equation*}
the elements being well defined since $\|e_k\|=1$, $k\in\N$ (see \eqref{ortho1}) and
\[
\left\{n(k)^{-(n(k)+1-\varepsilon_k)^{1/\beta'}}\right\}_{k=1}^\infty,
\left\{n(k)^{-\frac{1}{2}(n(k)+1-\varepsilon_k)
^{1/\beta'}}\right\}_{k=1}^\infty
\in l_1
\] 
($l_1$ is the space of absolutely summable sequences). 

Indeed, in view of \eqref{seq} and \eqref{eps},
for all $k\in\N$ sufficiently large so that 
\[
n(k)\ge 4^{\beta'},
\]
we have:
\begin{equation}\label{est}
n(k)^{-(n(k)+1-\varepsilon_k)^{1/\beta'}}\le n(k)^{-4}\le k^{-4}.
\end{equation}

In view of \eqref{ortho1}, by the properties of the \textit{s.m.},
\begin{equation}\label{vectors1}
E_A(\cup_{k=1}^\infty\Delta_k)f=f\ \text{and}\ E_A(\Delta_k)f=n(k)^{-(n(k)+1-\varepsilon_k)^{1/\beta'}}e_k,\ k\in\N.
\end{equation}
and
\begin{equation}\label{vectors2}
E_A(\cup_{k=1}^\infty\Delta_k)h=h\ \text{and}\
E_A(\Delta_k)h=n(k)^{-\frac{1}{2}(n(k)+1-\varepsilon_k)
^{1/\beta'}}e_k,\ k\in\N.
\end{equation}

For an arbitrary $t>0$ and any $g^*\in X^*$,
\begin{multline}\label{first1}
\int\limits_{\sigma(A)}e^{t|\lambda|^{1/\beta'}}\,dv(f,g^*,\lambda)
\hfill \text{by \eqref{vectors1};}
\\
\shoveleft{
=\int\limits_{\sigma(A)} e^{t|\lambda|^{1/\beta'}}\,d v(E_A(\cup_{k=1}^\infty \Delta_k)f,g^*,\lambda)
\hfill
\text{by \eqref{decompose};}
}\\
\shoveleft{
=\sum_{k=1}^\infty\int\limits_{\sigma(A)\cap\Delta_k}e^{t|\lambda|^{1/\beta'}}\,dv(E_A(\Delta_k)f,g^*,\lambda)
\hfill 
\text{by \eqref{vectors1};}
}\\
\shoveleft{
=\sum_{k=1}^\infty n(k)^{-(n(k)+1-\varepsilon_k)^{1/\beta'}}\int\limits_{\sigma(A)\cap\Delta_k}e^{t|\lambda|^{1/\beta'}}\,dv(e_k,g^*,\lambda)
}\\
\hfill
\text{since, by \eqref{or}, for $\lambda\in \Delta_k$,
$|\lambda| < n(k)+1-\varepsilon_k$;}
\\
\shoveleft{
\le\sum_{k=1}^\infty n(k)^{-(n(k)+1-\varepsilon_k)^{1/\beta'}}e^{t(n(k)+1-\varepsilon_k)^{1/\beta'}}\int\limits_{\sigma(A)\cap\Delta_k}1\,dv(e_k,g^*,\lambda)
}\\
\shoveleft{
=\sum_{k=1}^\infty e^{-\ln n(k)(n(k)+1-\varepsilon_k)^{1/\beta'}}e^{t(n(k)+1-\varepsilon_k)^{1/\beta'}}v(e_k,g^*,\Delta_k)
}\\
\shoveleft{
\le \sum_{k=1}^\infty e^{-(\ln n(k)-t)(n(k)+1-\varepsilon_k)^{1/\beta'}}v(e_k,g^*,\Delta_k)
\hfill
\text{by \eqref{tv};}
}\\
\shoveleft{
\le \sum_{k=1}^\infty e^{-(\ln n(k)-t)(n(k)+1-\varepsilon_k)^{1/\beta'}}4M\|e_k\|\|g^*\|
}\\
\hspace{1.2cm}
=4M\|g^*\|\sum_{k=1}^\infty e^{-(\ln n(k)-t)(n(k)+1-\varepsilon_k)^{1/\beta'}}
<\infty.
\hfill
\end{multline} 

Indeed, in view of \eqref{seq} and \eqref{eps}, for all $k\in\N$ sufficiently large so that
\[
\ln n(k)-t\ge 1,
\]
we have: 
\begin{equation*}
e^{-(\ln n(k)-t)(n(k)+1-\varepsilon_k)^{1/\beta'}
}\le e^{-k^{1/\beta'}}.
\end{equation*}

Similarly, for any $t>0$ and $n\in\N$,
\begin{multline}\label{second1}
\sup_{\{g^*\in X^*\,|\,\|g^*\|=1\}}
\int\limits_{\left\{\lambda\in\sigma(A)\,\middle|\,e^{t|\lambda|^{1/\beta'}}>n\right\}} 
e^{s|\lambda|^{1/\beta'}}\,dv(f,g^*,\lambda)
\\
\shoveleft{
\le 
\sup_{\{g^*\in X^*\,|\,\|g^*\|=1\}}\sum_{k=1}^\infty e^{-(\ln n(k)-t)(n(k)+1-\varepsilon_k)^{1/\beta'}}
\int\limits_{\left\{\lambda\in\sigma(A)\,\middle|\,e^{t|\lambda|^{1/\beta'}}>n\right\}\cap \Delta_k}1\,dv(e_k,g^*,\lambda) 
}\\
\shoveleft{
=\sup_{\{g^*\in X^*\,|\,\|g^*\|=1\}}\sum_{k=1}^\infty
e^{-\left(\frac{1}{2}\ln n(k)-t\right)(n(k)+1-\varepsilon_k)^{1/\beta'}}
e^{-\frac{1}{2}\ln n(k)(n(k)+1-\varepsilon_k)^{1/\beta'}}
}\\
\shoveleft{
\int\limits_{\left\{\lambda\in\sigma(A)\,\middle|\,e^{t|\lambda|^{1/\beta'}}>n\right\}\cap \Delta_k}1\,dv(e_k,g^*,\lambda)
}\\
\hfill
\text{since, by \eqref{seq} and \eqref{eps}, there is an $L>0$ such that}
\\
\hfill
\text{$e^{-\left(\frac{1}{2}\ln n(k)-t\right)(n(k)+1-\varepsilon_k)^{1/\beta'}}\le L$, $k\in\N$;}
\\
\shoveleft{
\le L\sup_{\{g^*\in X^*\,|\,\|g^*\|=1\}}\sum_{k=1}^\infty e^{-\frac{1}{2}\ln n(k)(n(k)+1-\varepsilon_k)^{1/\beta'}}
\int\limits_{\left\{\lambda\in\sigma(A)\,\middle|\,e^{t|\lambda|^{1/\beta'}}>n\right\}\cap \Delta_k}1\,dv(e_k,g^*,\lambda)
}\\
\shoveleft{
=L\sup_{\{g^*\in X^*\,|\,\|g^*\|=1\}}\sum_{k=1}^\infty {n(k)}^{-\frac{1}{2}(n(k)+1-\varepsilon_k)^{1/\beta'}}
\int\limits_{\left\{\lambda\in\sigma(A)\,\middle|\,e^{t|\lambda|^{1/\beta'}}>n\right\}\cap \Delta_k}1\,dv(e_k,g^*,\lambda)
}\\
\hfill
\text{by \eqref{vectors2};}
\\
\shoveleft{
=L\sup_{\{g^*\in X^*\,|\,\|g^*\|=1\}}\sum_{k=1}^\infty 
\int\limits_{\left\{\lambda\in\sigma(A)\,\middle|\,e^{t|\lambda|^{1/\beta'}}>n\right\}\cap \Delta_k}1\,dv(E_A(\Delta_k)h,g^*,\lambda)
}\\
\hfill
\text{by \eqref{decompose};}
\\
\shoveleft{
=L\sup_{\{g^*\in X^*\,|\,\|g^*\|=1\}} 
\int\limits_{\left\{\lambda\in\sigma(A)\,\middle|\,e^{t|\lambda|^{1/\beta'}}>n\right\}}1\,dv(E_A(\cup_{k=1}^\infty\Delta_k)h,g^*,\lambda)
}\\
\hfill
\text{by \eqref{vectors2};}
\\
\shoveleft{
= L\sup_{\{g^*\in X^*\,|\,\|g^*\|=1\}}\int\limits_{\left\{\lambda\in\sigma(A)\,\middle|\,e^{t|\lambda|^{1/\beta'}}>n\right\}}1\,dv(h,g^*,\lambda)
\hfill
\text{by \eqref{cond(ii)};}
}\\
\shoveleft{
\le L\sup_{\{g^*\in X^*\,|\,\|g^*\|=1\}}4M
\left\|E_A\left(\left\{\lambda\in\sigma(A)\,\middle|\,e^{t|\lambda|^{1/\beta'}}>n\right\}\right)h\right\|\|g^*\|
}\\
\shoveleft{
\le 4LM\left\|E_A\left(\left\{\lambda\in\sigma(A)\,\middle|\,e^{t|\lambda|^{1/\beta'}}>n\right\}\right)h\right\|
}\\
\hfill
\text{by the strong continuity of the {\it s.m.};}
\\
\hspace{1.2cm}
\to 4LM\left\|E_A\left(\emptyset\right)h\right\|=0,\ n\to\infty.
\hfill
\end{multline}

By Proposition \ref{prop} and \eqref{GC}, \eqref{first1} and \eqref{second1} jointly imply that
\begin{equation}\label{in}
f\in \bigcap_{t>0} D(e^{t|A|^{1/\beta'}})
={\mathscr E}^{(\beta')}(A).
\end{equation}

Let
\begin{equation}\label{functional1}
h^*:=\sum_{k=1}^\infty {n(k)}^{-2}e_k^*\in X^*,
\end{equation}
the functional being well defined since, by \eqref{seq}, $\{{n(k)}^{-2}\}_{k=1}^\infty\in l_1$ and $\|e_k^*\|=1$, $k\in\N$ (see \eqref{H-B1}).

In view of \eqref{H-B1} and \eqref{dist1}, we have:
\begin{equation}\label{funct-dist1}
\langle e_k,h^*\rangle=\langle e_k,{n(k)}^{-2}e_k^*\rangle=d_k {n(k)}^{-2}\ge \varepsilon {n(k)}^{-2},\ n\in\N.
\end{equation}

Fixing an arbitrary $0<\beta<\beta'$, for any $t>0$, we have:
\begin{multline}\label{nnotin}
\int\limits_{\sigma(A)}e^{t|\lambda|^{1/\beta}}\,dv(f,h^*,\lambda)
\hfill \text{by \eqref{decompose} as in \eqref{first1};}
\\
\shoveleft{
=\sum_{k=1}^\infty n(k)^{-(n(k)+1-\varepsilon_k)^{1/\beta'}}\int\limits_{\sigma(A)\cap\Delta_k}e^{t|\lambda|^{1/\beta}}\,dv(e_k,h^*,\lambda)
}\\
\hfill
\text{since, by \eqref{or} and \eqref{eps}, for $\lambda\in \Delta_k$,
$|\lambda| > n(k)-\varepsilon_k>0$;}
\\
\shoveleft{
\ge \sum_{k=1}^\infty e^{-\ln n(k)(n(k)+1-\varepsilon_k)^{1/\beta'}}e^{t(n(k)-\varepsilon_k)^{1/\beta}}v(e_k,h^*,\Delta_k)
\hfill
\text{by \eqref{eps};}
}\\
\shoveleft{
\ge \sum_{k=1}^\infty e^{t(k(n)-1)^{1/\beta}
-\ln k(n)(k(n)+1)^{1/\beta'}}|\langle E_A(\Delta_k)e_k,h^*\rangle|
\hfill
\text{by \eqref{funct-dist1};}
}\\
\shoveleft{
\ge \sum_{k=1}^\infty e^{t(k(n)-1)^{1/\beta}
-\ln k(n)(k(n)+1)^{1/\beta'}}\varepsilon {n(k)}^{-2}
}\\
\hspace{1.2cm}
=\sum_{k=1}^\infty \varepsilon e^{t(k(n)-1)^{1/\beta}
-\ln k(n)(k(n)+1)^{1/\beta'}-2\ln n(k)}=\infty.
\hfill
\end{multline} 

Indeed, for $k\ge 2$,
\begin{multline*}
t(n(k)-1)^{1/\beta}-\ln n(k)(n(k)+1)^{1/\beta'}-2\ln n(k)
\\
\shoveleft{
=(n(k)-1)^{1/\beta}
\left[t-\left[\dfrac{n(k)+1}{n(k)-1}\right]^{1/\beta'}
\dfrac{\ln n(k)}{(n(k)-1)^{1/\beta-1/\beta'}} 
-\dfrac{2\ln n(k)}{(n(k)-1)^{1/\beta}}\right]
}\\
\ \
\to \infty,\ k\to \infty,
\hfill
\end{multline*}

By Proposition \ref{prop} and \eqref{GC},
\eqref{nnotin} implies that, for each $0<\beta<\beta'$, 
\[
f\notin \bigcup_{t>0} D(e^{t|A|^{1/\beta}})
={\mathscr E}^{\{\beta\}}(A).
\]

Whence, in view of inclusions \eqref{incl1}, we infer that
\begin{equation}\label{notin}
f\not\in \bigcup_{0\le\beta<\beta'}{\mathcal E}^{\{\beta\}}(A).
\end{equation}

Comparing \eqref{in} and \eqref{notin}, we conclude that
\begin{equation*}
\bigcup_{0\le\beta<\beta'}{\mathcal E}^{\{\beta\}}(A)
\neq {\mathcal E}^{(\beta')}(A),
\end{equation*}
which completes the proof by contrapositive.
\end{proof}

\section{Main Result}

Lemma \ref{bbounded} affords a
rather short proof for the following  

\begin{thm}\label{mr}
Let $A$ be a scalar type spectral operator in a complex Banach space $(X,\|\cdot\|)$. If every weak solution of equation \eqref{1} belongs to the $\beta$th-order Roumieu type Gevrey class ${\mathscr E}^{\{\beta \}}\left((0,+\infty),X\right)$ with $0\le \beta<1$ (each one to its own), then the operator $A$ is bounded, and hence, all weak solutions of \eqref{1} are necessarily entire vector functions of exponential type.
\end{thm}

\begin{proof}\quad
Let $y(\cdot)$ be an arbitrary weak solution of equation \eqref{1}. Then, by the premise,
\[
y(\cdot) \in {\mathscr E}^{\{\beta\}}([0,\infty),X)
\] 
with some $0<\beta<1$ ($\beta$ depends on $y(\cdot)$).

This, by Proposition \ref{particular}, implies that
\begin{equation*}
y(t)\in {\mathscr E}^{\{\beta\}}(A),\ t\ge 0.
\end{equation*}
In particular,
\begin{equation*}
y(0)\in {\mathscr E}^{\{\beta\}}(A).
\end{equation*}

Since, by Theorem \ref{GWS}, the initial values of all weak solutions of equation \eqref{1} form the subspace
\begin{equation*}
\bigcap_{t\ge 0}D(e^{tA}),
\end{equation*}
in view of $D(e^{0A})=D(I)=X$, we have the inclusion
\begin{equation}\label{incl3}
\bigcap_{t>0}D(e^{tA})=\bigcap_{t\ge 0}D(e^{tA})\subseteq \bigcup_{0\le\beta<1}{\mathscr E}^{\{\beta\}}(A).
\end{equation}

Since, for an arbitrary $t>0$ and any $f\in X$, $g^*\in X^*$,
\[
\int\limits_{\sigma(A)}\left|e^{t\lambda}\right|\,dv(f,h^*,\lambda)
=
\int\limits_{\sigma(A)}e^{t\Rep\lambda}\,dv(f,h^*,\lambda)
\le
\int\limits_{\sigma(A)}e^{t|\lambda|}\,dv(f,h^*,\lambda)
\]
and, considering the inclusion
\[
\{\lambda\in\sigma(A)\,|\,e^{t\Rep\lambda}>n\}\subseteq \{\lambda\in\sigma(A)\,|\,e^{t|\lambda|}>n\},\ t>0,n\in\N,
\]
for any $n\in\N$,
\begin{multline*}
\int\limits_{\{\lambda\in\sigma(A)\,|\,\left|e^{t\lambda}\right|>n\}}\left|e^{t\lambda}\right|\,dv(f,h^*,\lambda)
=
\int\limits_{\{\lambda\in\sigma(A)\,|\,e^{t\Rep\lambda}>n\}}e^{t\Rep\lambda}\,dv(f,h^*,\lambda)
\\
\ \
\le
\int\limits_{\{\lambda\in\sigma(A)\,|\,e^{t|\lambda|}>n\}}e^{t\Rep\lambda}\,dv(f,h^*,\lambda)
\le
\int\limits_{\{\lambda\in\sigma(A)\,|\,e^{t|\lambda|}>n\}}e^{t|\lambda|}\,dv(f,h^*,\lambda),
\hfill
\end{multline*}
by Proposition \ref{prop}, we infer that, for each $t>0$,
\[
D(e^{tA})\supseteq D(e^{t|A|}),
\]
which, in view of \eqref{GC}, with $\beta=1$,
implies that
\begin{equation}\label{incl4}
\bigcap_{t>0}D(e^{tA})\supseteq
\bigcap_{t>0}D(e^{t|A|})={\mathscr E}^{(1)}(A),
\end{equation}
the equlity here being necessary and sufficient for $-A$ to be a generator of an analytic $C_0$-semigroup \cite{Markin2004(3)} (see also \cite{Hille-Phillips,Engel-Nagel}).

Observe that inclusion \eqref{incl4} also directly follows from the fact that ${\mathscr E}^{(1)}(A)$ is the subspace of initial values in $X$ generating the (classical) solutions of \eqref{1}, which are entire vector functions represented by power series
\eqref{powser} (see Preliminaries).

Inclusions \eqref{incl3} and \eqref{incl2} for $\beta'=1$ jointly with \eqref{incl4}, imply the following closed chain
\begin{equation*}
\bigcap_{t>0}D(e^{tA})\subseteq
\bigcup_{0\le\beta<1}{\mathscr E}^{\{\beta\}}(A)\subseteq
{\mathscr E}^{(1)}(A)\subseteq
\bigcap_{t>0}D(e^{tA}),
\end{equation*}
which proves that
\begin{equation*}
\bigcup_{0\le\beta<1}{\mathscr E}^{\{\beta\}}(A)=
{\mathscr E}^{(1)}(A).
\end{equation*}

Whence, by Lemma \ref{bbounded} with $\beta'=1$, we infer that the operator $A$ is \textit{bounded}, 
which completes the proof and implies that each weak solution $y(\cdot)$ of equation \eqref{1} is an entire vector function of the form
\[
y(z)=e^{zA}f=\sum_{n=0}^\infty \dfrac{z^n}{n!}A^nf,\ z\in \C,\ \text{with some}\ f\in X,
\] 
and hence, satisfying the growth condition
\[
\|y(z)\|\le \|f\|e^{\|A\||z|}, \ z\in \C,
\]
is of \textit{exponential type} (see Preliminaries).
\end{proof}

Theorem \ref{mr} generalizes {\cite[Theorem $5.1$]{Markin2001(2)}}, its counterpart for a normal operator $A$ in a complex Hilbert space and shows that,  while a scalar type spectral operator $A$ in a complex Banach space can be \textit{unbounded} while all weak solutions of equation \eqref{1} are \textit{entire} vector functions (see {\cite[Theorem $4.1$]{Markin2018(3)}} and {\cite[Corollary $4.1$]{Markin2018(3)}}), it cannot remain unbounded, if each weak solutions $y(\cdot)$ of \eqref{1}, in addition to being entire, is to satisfy the growth condition 
\begin{equation*}
\|y(z)\|\le Me^{\gamma|z|^{1/(1-\beta)}},\ z\in \C,
\end{equation*}
with some $0\le \beta<1$, $\gamma>0$, and $M>0$
depending on $y(\cdot)$ (see Preliminaries, cf. \eqref{order}).


\end{document}